\documentclass[12pt]{article}
\usepackage{amsmath,amsfonts,amssymb,amsthm,amscd}
\usepackage{xcolor}
\title{Picard-Vessiot extensions, linear differential algebraic groups and their torsors}
\date{\today}
\author{David Meretzky \and Anand Pillay\thanks{Supported by NSF grants DMS-1760212 and DMS-2054271}}

\newtheorem{Theorem}{Theorem}[section]
\newtheorem{Proposition}[Theorem]{Proposition}
\newtheorem{Definition}[Theorem]{Definition}
\newtheorem{Remark}[Theorem]{Remark}
\newtheorem{Lemma}[Theorem]{Lemma}

\newtheorem{Fact}[Theorem]{Fact}

\newtheorem{Question}[Theorem]{Question}

\begin{document}
\maketitle

\begin{abstract} Let $K$ be a differential field with algebraically closed field of constants $C_{K}$.  Let $K^{PV_{\infty}}$ be the (iterated) Picard-Vessiot closure
of $K$.  Let $G$ be a linear differential algebraic group over $K$, and $X$ a differential algebraic torsor for $G$ over $K$. We prove that $X(K^{PV_{\infty}})$ is
Kolchin dense in $X$. In the special case that $G$ is finite-dimensional we prove that $X(K^{PV_{\infty}}) = X(K^{diff})$  (where $K^{diff}$ is the differential 
closure of $K$). 

We   give close relationships between Picard-Vessiot extensions of $K$ and torsors for 
suitable finite-dimensional linear differential algebraic groups over $K$. 

We suggest some differential field analogues of the notion of ``boundedness" for fields (Serre's property (F)). 
\end{abstract}

\section{Introduction and preliminaries}
This paper is motivated by trying to refine or generalize two pieces of work. 

\vspace{2mm}
\noindent
(A) The first is Theorem 1.1 in \cite{Pillay-PV} by the second author, which says that an algebraically closed differential field $K$ has no proper Picard-Vessiot extensions 
if and only if Kolchin's constrained cohomology set $H_{\partial}^{1}(K, G)$ is trivial for any linear differential algebraic group $G$ defined over $K$, 
meaning that any torsor $X$ for a linear differential algebraic group $G$, all over $K$, has a $K$-point. 

\vspace{2mm}
\noindent
(B) The second is the theorem of Serre \cite{Serre} that a field $K$ (of characteristic $0$ for simplicity) is {\em bounded}, namely has finitely many extensions of degree $n$ for each $n$, 
if and only if the Galois cohomology set $H^{1}(K,G)$ is finite for every linear algebraic group $G$ over $K$. 

\vspace{2mm}
\noindent
At the end of \cite{Pillay-PV}, the second author asked explicitly for a differential field version of (B). This paper will touch on this last problem, but will mainly
be concerned with refinements of (A).  We will be giving an informal account of our results in this introduction, followed by  more precise definitions and references.

\vspace{2mm}
\noindent
In (A) an important case is when  $G$ is finite-dimensional.
Likewise in (B) a  first case is  when  $G$ is finite. See Proposition 8 in \cite{Serre} (Chapter 3, Section 4.1) written at the level of generality of profinite groups. 

\vspace{2mm}
\noindent
We will often refer to differential algebraic groups as $DAG$'s, linear differential algebraic groups as $LDAG$'s, and finite-dimensional linear differential algebraic groups
as $fdLDAG$'s. 

\vspace{2mm}
\noindent
This paper is about what we may call ``differential field arithmetic", that is properties of not necessarily differentially closed fields $K$,
especially regarding the properties of the set of $K$-rational points of differential algebraic varieties and groups over $K$. 

We could try to set up a kind of dictionary in passing from fields $K$ to differential fields $(K,\partial)$.
For example a (finite) linear algebraic group over $K$ is  generalized to a (finite-dimensional) $LDAG$ over $K$.
We would like to replace finite Galois extensions of $K$ by Picard-Vessiot ($PV$) extensions of $(K,\partial)$ (obtained from $K$ by 
adjoining a fundamental system of solutions of a homogeneous linear differential equation over $K$). The latter is problematic as a $PV$ extension
 of a $PV$ extension of $K$ is not (contained in) a $PV$ extension of $K$. This will be discussed later.

For an algebraic group $G$ over a field $K$, $H^{1}(K,G)$ classifies the torsors, 
also called principal homogeneous spaces ($PHS$'s),  
for $G$ over $K$ (in the category of algebraic varieties). Kolchin introduced a differential version which he called constrained
cohomology, and which we call differential Galois cohomology. So  for $(K,\partial)$ a differential field and $G$ a differential algebraic group over $K$, 
$H^{1}_{\partial}(K,G) = H^{1}_{\partial}(Aut_{\partial}(K^{diff}/K), G(K^{diff}))$ classifies the differential algebraic torsors over $K$ for $G$. 
In \cite{Kamensky-Pillay}  finiteness results for Galois cohomology  were used used to obtain existence theorems for Picard-Vessiot extensions  with specified properties.  
One of the interests in possible finiteness results for differential Galois cohomology was to extend these existence results to 
{\em parametrized} Picard-Vessiot extensions (in the presence of  several derivations), and  in \cite{Leon-Sanchez-Pillay} such an extension was given.  
However the differential fields $(K,\partial)$ for which finiteness of differential Galois cohomology was proved in 
\cite{Leon-Sanchez-Pillay} were fairly special, namely differentially large in the sense of \cite{LST} and bounded as fields, 
in particular having {\em  no} proper Picard-Vessiot extension with 
 connected differential Galois group.  Moreover for such $K$   $H^{1}_{\partial}(K,G)$ for 
$G$ a $LDAG$ over $K$ essentially reduced to $H^{1}(K,H)$ for an associated linear algebraic group $H$ over $K$, which is finite by Serre's theorem. 

We will see that as soon as a differential field has some rather mild Picard-Vessiot extensions
 then one will have infiniteness of both differential Galois cohomology sets and Picard-Vessiot extensions with a given Galois group.  
So finiteness conditions may be not the appropriate notions (outside very special situations such as discussed in the last paragraph). 
We will offer some tentative suggestions for notions of ``smallness" of both the family of Picard-Vessiot extensions of $K$ and of $H^{1}_{\partial}(K,G)$ for $G$ a $fdLDAG$ over $K$. 

We will assume the differential field $K$ to  have algebraically closed field $C_{K}$ of constants. $K^{diff}$ denotes a fixed differential closure of $K$. 
All differential Galois theory over $K$ goes on inside $K^{diff}$. 
Inside $K^{diff}$ we have various distinguished differential subfiields. First   $K^{PV}$ is the union of all $PV$-extensions of $K$. 
Define inductively, $K^{PV_{1}} = K^{PV}$ and $K^{PV_{n+1}} =   (K^{PV_{n}})^{PV}$. And let $K^{PV_{\infty}} = \cup_{n}K^{PV_{n}}$, the (iterated) Picard-Vessiot closure of $K$.  
So to say that $K$ has no proper Picard-Vessiot extensions means that $K = K^{PV} = K^{PV_{\infty}}$.  The first result is a refinement of (A):
\begin{Proposition} (i) Let $G$ be a $fdLDAG$ over $K$ and $X$ a differential algebraic torsor for $G$ over $K$,  then $G(K^{diff}) = G(K^{PV_{\infty}})$ and $X(K^{diff}) = X(K^{PV_{\infty}})$.
\newline
(ii) For $G$ a (possibly infinite-dimensional) $LDAG$ over $K$, and $X$ a differential algebraic torsor over $K$, $G(K^{PV_{\infty}})$ is Kolchin-dense in $G$ and $X(K^{PV_{\infty}})$ is 
Kolchin-dense in $X$.
\end{Proposition}

This of course implies Theorem 1.1 of \cite{Pillay-PV}. Note that each of (i), (ii) for $X$ gives the same result for $G$ (as $G$ is in bijection with $X$ over $K$ together with any 
point of $X$).  Similarly,  modulo Theorem
1.1 of \cite{Pillay-PV}, each of (i), (ii) for $G$ implies it for $X$.   One can also view Proposition 1.1 (ii) for $G$ as a differential analogue of the fact that for any characteristic $0$ 
field $K$ and connected linear algebraic group $G$ over $K$, $G(K)$ is Zariski-dense in $G$. 

We will discuss Kolchin density in subsection 1.4 below.


It already follows from Theorem 1.1 of \cite{Pillay-PV} (and exact sequences of cohomology) that  $H^{1}_{\partial}(K,G)$ is canonically
isomorphic to $H^{1}_{\partial}(Aut_{\partial}(K^{PV_{\infty}}/K), G(K^{PV_{\infty}}))$ for any  $LDAG$ $G$ over $K$. But  Proposition 1.1 (i)
implies the stronger statement in the case that $G$ is a $fdLDAG$ over $K$, that we have {\em equality} at the level of cocycles.

We do not manage to find a suitable extension of (B) to the differential context, but we do give some closer connections 
between the class of $PV$-extensions of $K$ and  torsors for suitable $fdLDAG$'s over $K$, which gives some partial generalizations of (B) with respect to 
suitable notions of ``smallness", as well as refinements of (A).

The Picard-Vessiot theory concerns homogeneous linear differential equations over a differential field $K$.
We start with such a linear differential equation over $K$ in vector form 
\newline
(*) $\partial Y = AY$, 
\newline
$Y$ being a $n\times 1$ column vector
of unknowns, and $A$ an $n\times n$ matrix with coefficients from $K$. 

 A ``fundamental system of solutions" of (*) is a nonsingular $n\times n$ matrix $Z$ (in $GL_{n}(K^{diff})$) such that $\partial Z = AZ$. The $PV$ extension of $K$
corresponding to the equation (*) is  the (differential) field $K(Z)$.  It depends only on the equation (*), not the choice of $Z$.
The differential algebraic  locus of $Z$ over $K$ 
will be a left torsor $X$ over $K$ for the ``intrinsic" Galois group of (*), a differential algebraic subgroup $G$ of $GL_{n}(K^{diff})$ over $K$.  
So to the linear differential equation  $\partial Y = AY$ we obtain a differential algebraic torsor $(G,X)$, depending only on the equation (*), as well as a Picard-Vessiot extension $L$ of $K$.

\begin{Proposition}  Let $\partial Y_{1} = A_{1}Y_{1}$ and $\partial Y_{2} = A_{2}Y_{2}$ be linear differential equations over $K$ with corresponding torsors  
$(G_{1}, X_{1})$ and $(G_{2}, X_{2})$ and Picard-Vessiot extensions $L_{1}$, $L_{2}$. Then $L_{1} = L_{2}$ if and only if  $(G_{1}, X_{1})$ and $(G_{2}, X_{2})$ are isomorphic as
differential algebraic torsors over $K$.
\end{Proposition}

\begin{Proposition} Let $(G,X)$ be a differential algebraic right torsor over $K$ where $G$ is a fdLDAG ``internal" to the constants $C$
 (of an ambient differentially closed field such as $K^{diff}$).  Let $L_{G}$ be the differential field generated over $K$ by $G(K^{diff})$ and likewise 
$L_{X}$ the differential field generated over $K$ by $X(K^{diff})$. Then $K_{G}\subseteq K_{X}$, both are Picard-Vessiot extensions of $K$, and the 
differential Galois group of $L_{X}$ over $L_{G}$ differentially algebraically over $L_{X}$ embeds in $G(K^{diff})$.
\end{Proposition} 

One notion of smallness that we will discuss is about suitable collections of $PV$ extensions being ``of finite type" in the sense of generating a $PV$-extension of $K$. Likewise
for a suitable $fdLDAG$ $G$ over $K$, the collection of torsors $X$ for $G$ (or rather their sets of $K^{diff}$ points) should  generate a $PV$-extension of $K$.  This is discussed in Section 2.3. 

In the remainder of this section we will give proper references and  definitions.
This paper is about differential algebraic geometry, in the sense of Kolchin, and differential Galois theory.   The model theory of differentially closed fields is very closely related to, even synonymous,  with differential algebraic geometry. Model theory also provides generalizations of the  differential Galois theories of Picard-Vessiot and Kolchin, although in this paper we will only be concerned with the Picard-Vessiot theory. 

The reader can consult both the introduction and preliminaries section of \cite{Pillay-PV}, as well as \cite{Pillay-foundational}. 
The latter gives a comprehensive account of Kolchin's differential algebraic geometry and the connection to model theory. As mentioned in \cite{Pillay-PV}, the model theoretic machinery subsumes that of Kolchin but provides a certain general point of view as well as new technical  tools. So here we will mainly fix notation, although we there are a few repetitions of material from \cite{Pillay-PV} and \cite{Kamensky-Pillay}. In the meantime recall that when it comes to groups and their homogeneous spaces, differential algebraic objects coincide with definable objects, including with respect to differential fields of definition. 

We let $K = (K,\partial)$ be a differential field of characteristic $0$ whose field of constants is denoted $C_{K}$.  We will often suppress the symbol for the derivation in both $K$ and differential fields in which $K$ embeds.
We let ${\mathcal U}$ be a ``universal domain" in the sense of Kolchin, which model-theoretically is a $\kappa$-saturated differentially closed field of cardinality $\kappa$ which contains $K$ and where $\kappa > |K|$.  Such ${\mathcal U}$ exists due to $\omega$-stability of the theory $DCF_{0}$ of differentially closed fields in the language of differential unitary rings.  ${\mathcal C}$ denotes the field of constants of ${\mathcal U}$. 
We fix a differential (or as Kolchin says ``constrained") closure $K^{diff}$ of $K$ inside ${\mathcal U}$. 
Differential algebraic varieties, namely solution sets of finite systems of differential polynomial equations in $n$ 
differential indeterminates, will be considered as point sets in ${\mathcal U}^{n}$. For $X$ such an object and $L$ a differential subfield of ${\mathcal U}$ over which $X$ is defined, $X(L)$ denotes the points of $X$ in $L^{n}$.
Likewise definable will mean definable in the structure ${\mathcal U}$. 

Given $K$ and a subset $A$ of ${\mathcal U}$, $K\langle A \rangle$ denotes the differential field generated over $K$ by $A$.

We will mainly work inside $K^{diff}$.

\vspace{2mm}
\noindent
{\bf Assumption.}  $C_{K}$ is algebraically closed. 

\vspace{2mm}
\noindent
This assumption is so as to have the existence of Picard-Vessiot extensions of any linear DE over $K$.  Note that the assumption is NOT satisfied for the differential subfields $L$ of models of the theories such as $CODF$ of differential fields studied in \cite{Leon-Sanchez-Pillay}, although in this case we still have the existence of ``nice" PV-extensions. 
Anyway it follows from the assumption above  that $C_{K} = C_{K^{diff}}$.

\vspace{2mm}
\noindent
We wish to thank Michael Wibmer for various communications about the material in this paper as well as some questions, in particular about
the Kolchin denseness of $G(K^{PV_{\infty}})$ which we prove.   Thanks also to Omar Leon Sanchez for his
invitation to the second author to Manchester in the summer of 2023 where some of the work on, and writing of, the paper was done. 

\subsection{Linear differential algebraic groups and their torsors}

The notion of a differential algebraic group ($DAG$) over a differential field is due to Kolchin \cite{Kolchin-DAG}, and is roughly defined as a differential algebraic variety $X$
equipped with a group structure which is a morphism $X\times X \to X$ in the sense of differential algebraic varieties.
In \cite{Pillay-foundational} we point out equivalent descriptions such as (i) groups definable in ${\mathcal U}$ and (ii) definable (in ${\mathcal U}$) subgroups of algebraic groups. 
A $DAG$ over $K$ is one defined with parameters from $K$. 

Likewise a differential algebraic torsor for a $DAG$ $G$ can be defined as a definable set $X$ with a definable action of $G$ on $X$ which is regular (strictly
transitve).  We can have left actions or right actions, in which case we talk about left torsor or right torsors.  From a left torsor we get a right torsor by defining
$x*g = g^{-1}x$.   Again definability of a torsor $(G,X)$ over $K$ means that all the data are definable over $K$. 

\begin{Definition} (i) Let $(G_{1}, X_{2})$ and $(G_{2}, X_{2})$ be differential algebraic (left) torsors over $K$. We will say that they are isomorphic over $K$,
if there are a $K$-definable group isomorphism $i:G_{1}\to G_{2}$ and $K$-definable bijection $j:X_{1}\to X_{2}$,
such that for each $g\in G_{1}$ and $x\in X_{1}$, $j(gx_{1}) = i(g)j(x_{1})$. 
\newline
(ii) Suppose $X_{1}$ and $X_{2}$ are (left) differential algebraic torsors for the same $DAG$ $G$, all defined over $K$.
We say that $X_{1}$ and $X_{2}$ are $G$-isomorphic over $K$, if there are $i:G\cong G$, and $j:X_{1}\to X_{2}$ as in (i) above such that $i$ is the identity.
Namely there is a bijection $j:X_{1}\to X_{2}$ defined over $K$ which commutes with the given actions of $G$. 
\end{Definition}

Likewise for right torsors. Kolchin's constrained cohomology is an adaptation of the Galois cohomology of algebraic groups to differential fields and differential algebraic groups. In \cite{Pillay-Galois}, this was generalized to a model-theoretic setting and we will freely use results and notation from that paper.  We give a brief summary, following notation from \cite{Pillay-Galois}.  Let us fix a $DAG$ $G$ over $K$. 

$H^{1}_{\partial}(K,G)$, or more precisely $H_{\partial}(Aut_{\partial}(K^{diff}/K), G(K^{diff}))$ is the set of $K$-definable (left) cocycles from ${\mathcal G} = Aut_{\partial}(K^{diff}/K)$ to $G(K^{diff})$, 
up to the equivalence relation of being cohomologous.   
Here a cocycle is a crossed homomorphism $f$ from ${\mathcal G} = Aut_{\partial}(K^{diff}/K)$ to $G(K^{diff})$, 
namely $f(\sigma\tau) = f(\sigma)\sigma(f(\tau))$ for all $\sigma, \tau \in {\mathcal G}$.  Definability over $K$ means that for some $K$-definable function $h(-,-)$, 
and some tuple $a$ from $K^{diff}$, $f(\sigma) = h(a,\sigma(a))$ for all $\sigma\in {\mathcal G}$.  Finally, cocycles $f$ and $g$ are cohomologous if  
$f(\sigma) = b^{-1}g(\sigma) \sigma(b)$ for some $b\in G(K^{diff})$ and all $\sigma\in {\mathcal G}$.  By a trivial cocycle we mean something of the form $b^{-1}\sigma(b)$ for some $b\in G(K^{diff})$. 

\begin{Fact} There is a natural isomorphism (of pointed sets) between the set of differential algebraic right torsors for $G$ over $K$, up to $G$-isomorphism over $K$, and $H^{1}_{\partial}(K,G)$.
\end{Fact}

Analogously for left torsors.
Note that when $G(K^{diff}) = G(K)$, then ${\mathcal G}$ acts trivially on $G(K^{diff})$ and so a $K$-definable cocycle is just a $K$-definable homomorphism from ${\mathcal G}$ to $G(K^{diff})$.  If moreover $G$ is commutative then  two such homomorphisms are cohomologous iff they are equal.   It follows, for example, that

\begin{Lemma} Let $G$ be the additive group of the constants, and suppose that $H^{1}_{\partial}(K,G)$ is nontrivial (namely there is a torsor for $G$ over $K$ with no $K$-point). Then
$H^{1}(K,G)$ is infinite. 
\end{Lemma} 
\begin{proof} So $G(K^{diff}) = (C_{K^{diff}},+) = (C_{K},+)$. Now $(C_{K},+)$  has infinitely many $K$-definable automorphisms, 
namely multiplication by any nonzero element of $C_{K}$. Composing the $K$-definable homomorphism corresponding to the torsor $X$ by these 
automorphsms gives infinitely many distinct $K$-definable homomorphisms from ${\mathcal G}$ to $G(K^{diff})$ so yielding infiniteness of $H^{1}(K,G)$ by above remarks.
\end{proof} 

We will be focusing on linear differential algebraic groups, namely differential algebraic (equivalently definable) subgroups of $GL_{n}({\mathcal U})$ for some $n$.
We call these $LDAG$'s. An important further special case is when $G$ is finite-dimensional. Remember that a $K$-definable set $X$ is finite-dimensional if there is a 
finite bound on the transcendence degree of $K\langle a\rangle$ over $K$ for $a\in X$.  We call finite-dimensional $LDAG$'s, $fdLDAG$'s

\subsection{Galois theory of linear differential equations}
By a (homogeneous) linear differential equation over $K$ in vector form we mean something of the form 
\newline
(*) $\partial Y = AY$ where $Y$ is an $n\times 1$ column vector of unknowns and $A$ is an $n\times n$ matrix over $K$.  

The solution set of (*)  in ${\mathcal U}$  is a $n$-dimensional vector space ${\mathcal Y}$   
over ${\mathcal C}$.  Likewise the solution set
in $K^{diff}$ is an $n$-dimensional vector space over $C_{K^{diff}} = C_{K}$. By a {\em fundamental system of
solutions of (*)} we mean a basis for ${\mathcal Y}(K^{diff})$ over $C_{K}$.  The Picard-Vessiot extension of $K$ for the equation (*)
is precisely the (differential) field generated over $K$ by such a fundamental system. Equivalently it is the (differential) field extension of $K$ generated
by the full set ${\mathcal Y}(K^{diff})$ of solutions of ${\mathcal Y}$ in $K^{diff}$. 

It is an elementary fact that $n$ solutions $b_{1},..,b_{n}$ of (*) are ${\mathcal C}$-linearly independent iff the $n\times n$ matrix
${\bar b} = (b_{1} ... b_{n})$ is nonsingular.  Hence looking for a fundamental system of solutions of (*) is the same thing as
looking for a solution of the {\em matrix} equation on $GL_{n}$, $\partial Z = AZ$ in $GL_{n}(K^{diff})$. 
Again, with ${\mathcal Z}$ the set of (nonsingular) solutions of $\partial Z = AZ$, the Picard-Vessiot extension  of $K$ for (*)
is generated over $K$ by ${\mathcal Z}(K^{diff})$.

Let us now fix the equation 
\newline
(*)     $\partial Y = AY$  over $K$.

Let us fix ${\bar b}\in {\mathcal Z}(K^{diff})$ as above, and $L = K({\bar b})$.  Let $\phi({\bar x})$ be a formula over $K$ isolating $tp({\bar b}/K)$.  Let ${\mathcal X}$ be the set of solutions of the formula $\phi({\bar x})$, so a $K$-definable subset of ${\mathcal Z}$. There are two differential algebraic incarnations of the (differential) Galois group of the equation (*).  

\begin{Definition} (i) By the intrinsic differential Galois group of (*) , we mean $H^{+} = \{{\bar b}_{1}{\bar b}^{-1}: {\bar b_{1}}$ realizes $\phi({\bar x})$  in ${\mathcal U}\}$  (multiplication in $GL_{n}$)
\newline 
(ii) By the extrinsic differential Galois group of (*), we mean $H_{\bar b} = \{{\bar b}^{-1}{\bar b}_{1}: {\bar b}_{1}$ realises $\phi({\bar x})$ in ${\mathcal U}$\} (Again multiplication in $GL_{n}$). 
\end{Definition} 

\begin{Remark} (i)
The choice of language is influenced by Daniel Bertrand and his book review \cite{Bertrand-Magid} which is itself a nice survey of the Picard-Vessiot theory.  In particular the Zariski closure of $H^{+}$ is 
what Bertrand calls the Katz group. 
\newline
(ii) As stated $H^{+}$ and $H_{{\bar b}}$ are definable subsets (in fact subgroups) of $GL_n({\mathcal U})$. 
\end{Remark}

Let us write $Gal(L/K)$ for the group of automorphisms of the differential field $L$ which fix $K$ pointwise.

\begin{Lemma}(i)   The intrinsic Galois group $H^{+}$  does not depend on the choice of ${\bar b}\in {\mathcal Z}(K^{diff})$. It is a finite-dimensional 
$K$-definable subgroup of $GL_{n}({\mathcal U})$.  $H^{+}$ acts on the left (matrix multiplication) on
 the set ${\mathcal X}$ of solutions of $\phi({\bar x})$ in ${\mathcal U}$. As such ${\mathcal X}$ is a left torsor for $H^{+}$ over $K$.
The map $\rho$ taking $\sigma\in Gal(L/K)$ to $\sigma({\bar b}){\bar b}^{-1}$ establishes an isomorphism between $Gal(L/K)$ and $H^{+}(K^{diff})$.  
Moreover $Gal(L/K)$ acts naturally on ${\mathcal X}(K^{diff})$ and
 $\rho$ establishes an isomorphism between this action and the left torsor $(H^{+}(K^{diff}), {\mathcal X}(K^{diff}))$. 
\newline
(ii) $H_{\bar b}$ is a subgroup of $GL_{n}({\mathcal C})$ defined over $C_{K}$. $H_{b}$ acts by matrix multiplication on ${\mathcal X}$ on the right, 
yielding a right torsor.
The map ${\omega}_{\bar b}$ taking $\sigma\in Gal(L/K)$ to ${\bar b}^{-1}\sigma({\bar b})$ gives an isomorphism 
between $Gal(L/K)$ and $H_{\bar b}(C_{K^{diff}}) = H_{\bar b}(C_{K})$, which depends on $\bar b$.   Composing $
\rho^{-1}$ with $\omega_{\bar b}$ determines an isomorphism between $H^{+}(K^{diff})$ and $H_{\bar b}
(K^{diff})$ defined over $L = K({\bar b})$.
\newline
(iii) When $H^{+}$ is commutative it coincides with $H_{\bar b}$. 
\end{Lemma} 
\begin{proof}  Well-known, but see \cite{Kamensky-Pillay}. 
\end{proof}

\begin{Lemma} Suppose $K$ has a Picard-Vessiot extension $L$ with Galois group $(C_{K},+)^{2}$. Then $K$ has infinitely many $PV$ extensions with Galois group (intrinsic or extrinsic) isomorphic over 
$K$ to $(C_{K},+)$.
\end{Lemma}
\begin{proof} For each nonzero $\alpha\in C_{K}$, the graph of multiplication by $\alpha$ is an algebraic subgroup $G_{\alpha}$ of $(C_{K},+)^{2}$.
Let $L_{\alpha}$ be the elements of $L$ fixed by $G_{\alpha}$. Then the $L_{\alpha}$ are distinct $PV$ extensions of $K$ with Galois group $(C_{K},+)$. 

\end{proof}

\begin{Definition}   Let $K^{PV}$ be the (differential) field extension of $K$ generated over $K$ by all Picard-Vessiot extensions of $K$  (inside our fixed choice of $K^{diff}$). 
\end{Definition} 

\begin{Remark} (i) $K^{PV}$ is the union of all $PV$ extensions of $K$.
\newline
(ii) $K^{diff}$ is also a differential closure of $K^{PV}$.
\newline
(ii) $K^{PV}$ is fixed setwise by ${\mathcal G} = Aut_{\partial}(K^{diff}/K)$  (even by  $Aut_{\partial}({\mathcal U}/K)$).
\end{Remark}
\begin{proof} (i) Let $a\in K^{PV}$ so $a\in dcl(K,b_{1},..,b_{n})$ with $b_{i}\in L_{i}$ a $PV$ extensions of $K$. 
Without loss $L_{i} = K(b_{i})$. But then  $K(b_{1},..,b_{n})$ is also a $PV$ extension of $K$ and contains $a$.
\newline
(ii) is left to the reader, but it is basic model theory and is not really needed.
\newline
(iii) is obvious as each $PV$ extension of $K$ has that property. 
\end{proof} 

$Aut_{\partial}(K^{PV}/K)$ has the structure of  a proaffine algebraic group in (or over) the algebraically closed field $C_{K}$.  
This is the extrinsic proalgebraic differential Galois group, as in Definition 1.7. The structure of this group for base differential fields such as 
$({\mathbb C}(t), d/dt)$ is studied in several papers such as \cite{Bachmayr}.

\begin{Definition} (i) Define inductively  $K^{PV_{1}} = K^{PV}$ and  $K^{PV_{n+1}} = (K^{PV_{n}})^{PV}$.
\newline
(ii) Define $K^{PV_{\infty}} = \cup_{n}K^{PV_{n}}$, and call it the Picard-Vessiot closure of $K$. 
\end{Definition}

Note that in  \cite{MagidI} $K^{PV}$ is called the Picard-Vessiot closure, and in \cite{MagidII} $K^{PV_{\infty}}$ is called the complete Picard-Vessiot closure. 

\begin{Remark} Again each $K^{PV_{n}}$ and thus also $K^{PV_{\infty}}$ are ${\mathcal G}$-invariant, and $K^{diff}$ is a differential closure of $K^{PV_{\infty}}$. 
\end{Remark}

\begin{Lemma} $K^{PV_{\infty}}$ is algebraically closed and Picard-Vessiot closed (in the sense of having no proper Picard-Vessiot extensions).
\end{Lemma}
\begin{proof} This is obvious, but we will say a few words.  First note that as $K^{PV_{\infty}}\leq K^{diff}$ and $C_{K} = C_{K^{diff}}$ then the constants of $K^{PV_{\infty}}$ are algebraically closed, whereby any finite Galois extension of $K^{PV_{\infty}}$ is a PV extension (see \cite{Singer-vanderPut}).  So it suffices to prove that $K^{PV_{\infty}}$ has no proper $PV$ extension.
Let $\partial Y= AY$ be a linear DE over $K^{PV_{\infty}}$. Then the entries of $A$ are all in $K^{PV_{m}}$ for some $m$. But then a fundamental matrix for the equation can be found over $K^{PV_{m+1}}$ so over $K^{PV_{\infty}}$. 
\end{proof}

\subsection{Internality} 

\begin{Definition} Let $X$ be a $K$-definable set (or differential algebraic variety defined over $K$). We say that $X$ is {\em internal to the constants} or {\em internal to $\mathcal C$}  if there is a definable (with parameters, not necessarily from $K$) bijection of $X$ with a definable subset $Y$ of ${\mathcal C}^{n}$ some $n$. 
\end{Definition}

Internality to the constants  implies finite-dimensionality.

\begin{Lemma} Let $G$ be a $fdLDAG$ over $K$.  Suppose $G$ is internal to the constants. Then the differential field $L$ generated over $K$ by $G(K^{diff})$ is a Picard-Vessiot extension of $K$ . 
Moreover there is a definable over $L$ isomorphism of $G(K^{diff})$ with an algebraic subgroup $H$ of some $GL_{n}(C_{K})$. 
\end{Lemma}
\begin{proof} We  may assume that $G$ is connected.  The internality assumption  implies that $G$ is definably isomorphic over some additional parameters to an algebraic group $H$ living in the constants ${\mathcal C}$.
We can find such an isomorphism defined over $K^{diff}$, whereby $H$ is defined over $C_{K}$.  
By Proposition 4.1 of \cite{Pillay-DGTIV}, there is an isomorphism  $\rho$  of $G$ with $H$  defined over a Picard-Vessiot extension $L_{1}$ of $K$.  It follows that  $\rho(G(K^{diff})) = H(C_{K})$ 
whereby $G(K^{diff}) = G(L_{1})$.
Let $L$ be the differential subfield of $L_{1}$ generated by the points of $G(K^{diff})$. Then $L$ is invariant under $Aut_{\partial}(L_{1}/K)$ so is a Picard-Vessiot extension of $K$.
On general model-theoretic grounds $\rho$ can be defined over $L$. 

Finally, we want to show that $G$ being a {\em linear} differential algebraic group, implies that $H$ is  (definably over $C_{K}$) isomorphic to a linear algebraic group in the constants. 
This is because, if $H$ were not linear there would be a $C_{K}$-definable surjective homomorphism from $H$ to an abelian variety $A$ living in the constants.
The same would be then true of $G$.  However the proof of Case (1) of Theorem 1.1 in \cite{Pillay-PV} gives that $G$ has a composition series of definable subgroups where each quotient is definably
isomorphic to a (connected) linear algebraic group living in the constants. But there is no nontrivial algebraic homomorphism from a linear algebraic group to an abelian variety. This gives us
the required contradiction.   (This last argument is close to the material in Section 2.1 of the current paper.)
\end{proof}

\begin{Remark} (i) The first part of the proof above shows that $G(K^{diff})$ generates a $PV$ extension of $K$ when $G$ is any $DAG$ internal to the constants. 
(ii) The converse to 1.17, ,  that $G(K^{diff})$ generating a $PV$ extension of $K$ implies $G$ being internal to the constants, is false, as can be seen by taking $K = K^{diff}$ and suitable $G$. 
\end{Remark}

Let us give an example where $K$ is a field of constants, so very far from being differentially closed, $G$ is a $fdLDAG$ over $K$, $G$ is not internal to the constants, but $G(K^{diff}) = G(K^{PV})$

Recall $dlog$ is the surjective homomorphism from ${\mathcal U}^{\times}$ to $({\mathcal U},+)$, taking $x$ to $(\partial x).x^{-1}$.
Let $G = dlog^{-1}({\mathcal C})$, defined by  $\partial(\partial(x)/x)) = 0$.   So $G$ is a finite-dimensional $\emptyset$-definable connected subgroup of ${\mathcal U}^{\times}$. It is well-known that $G$ is not internal to the constants.  $dlog$ gives rise to a short exact sequence  ${\mathcal C}^{\times} \to G \to ({\mathcal C}, +)$. 
Let $K$ be some algebraically closed field of constants (so $K = C_{K}$).  We try to describe $G(K^{diff})$. Note that $dlog(G(K^{diff}) = (K,+)$. For each $a\in K$, $dlog^{-1}(a)(K^{diff})$ generates a PV extension $L_{a}$ of $K$ with Galois group $K^{\times}$. So $G(K^{diff}) = \cup_{a\in K}G(L_{a}) =  G(K^{PV})$.  One can check that  $\cup_{a}L_{a}$ is not contained in a single PV extension of $K$ (otherwise $G$ would be internal to the constants).

Let us mention  briefly a related example. Let $K$ be an algebraically closed field of constants. Let $c\in K$ be nonzero, and consider the equation $\partial(\partial(x)/x)) = c$. This defines a coset $X$ (in the multiplicative
group) of the group $G$ above.  Let $a$ satisfy the equation (in $K^{diff})$, and let $b = \partial(a)/a$.  Then $\partial(b) = c$ and $K(b)$ is a $PV$ extension $L_{1}$ of $K$ (with Galois group the additive group of the constants).
Moreover  $a$ generates over $L_{1}$ a $PV$ extension $L_{2}$ of $L_{1}$  (with Galois group the multiplicative group of the constants).  We see that the $K$-definable set $X$ determines a complete type over $K$ (i.e. $tp(a/K)$). Note
that $a\notin K^{PV}$, namely $X$ has no point in $K^{PV}$. For otherwise $X$ would be internal to the constants, but then so would be $G$, a contradiction.

\begin{Question} Suppose $K = C_{K}$ is an algebraically closed field of constants, $G$ is a $fdLDAG$ over $K$, and $G(K^{diff})$ generates over $K$ a PV extension of $K$. Must $G$ be internal to the constants. 
\end{Question} 

We briefly recall strongly normal extensions. According to Kolchin \cite{Kolchin}, a differential field extension $L$ of $K$ is strongly normal, if $C_{L} = C_{K}$,  $L$ is finitely generated over $K$ (as a 
differential field) and for a finite tuple $b$ generating $L$ over $K$, whenever $c$ realizes $tp(b/K)$ then $c\in dcl(K,{\mathcal C}, b)$.  This implies that $L$ is even finitely generated as a field, and 
$L\leq K^{diff}$. Moreover $Aut_{\partial}(L/K)$ has the structure of a finite dimensional differential algebraic group (but not necessarily linear). 
And as in the PV case there are two incarnations of this "definable Galois group"; (a) as a $K$-definable group with a $K$-definable action on the set $X$ of realizations of $tp(b/K)$ (for $b$ a generator of $L$ over $K$) isomorphic to the action of $Aut_{\partial}(L/K)$ on $X$, and (b) as an algebraic group in the (algebraically closed field of constants).  These two incarnations are isomorphic, definablly over $b$. 
See also \cite{Kamensky-Pillay} for an account of the strongly normal theory, at least when $K$ is algebraically closed.

A basic result of Kolchin \cite{Kolchin} is:
\begin{Fact} Suppose $L$ is a strongly normal extension of $K$ and the definable Galois group of $L$ over $K$ is definably isomorphic to a fdLDAG (or to a linear algebraic group living in $C_{K}$). Then $L$ is a PV extension of $K$. 
\end{Fact}

\begin{Lemma} Suppose $X$ is a $K$-definable set which is internal to the constants. Then $X(K^{diff})$ generates over $K$ a strongly normal extension of $K$.
\end{Lemma}
\begin{proof} Let $c$ be a finite tuple from $K^{diff}$ such that there is definable over $c$ bijection between $X$ and a $C_{K}$-definable subset of ${\mathcal C}^{n}$  (some $n$).  
On general grounds we  can rechoose $c$ to be a finite tuple of elements of $X(K^{diff})$. Then $c$ generates a strongly normal extension $L$ of $K$ and clearly $X(K^{diff})$ is contained in $L$. 
\end{proof} 

\subsection{Kolchin density}

We recall briefly the Kolchin topology on differential algebraic varieties, as well as the meaning for differential algebraic groups. The reader is referred to \cite{Pillay-foundational}
for more details and background behind the facts stated here. 

We will restrict attention to differential algebraic subsets of the affine spaces ${\mathcal U}^{n}$, namely sets defined by a finite system of differential 
polynomial equations in $n$ differential indeterminates with coefficients from ${\mathcal U}$.  Let $X\subseteq {\mathcal U}^{n}$ be such.
The Kolchin topology on $X$ is that whose closed sets are given by solution sets of finite systems of differential polynomial equations. This is a topology because of the $DCC$ or
Noetherian condition
(no infinite descending chain of closed sets).  It follows that $X$ is a finite union of its (Kolchin) irreducible components, each of which is defined over the algebraic closure of the
(differential) field of definition of $X$. 

An abstract subset $Z$ of $X$ is defined to be Kolchin dense in $X$ if if $X$ is the Kolchin closure of $Z$, namely there is no proper closed subset $Y$ of $X$ with
$Z\subseteq Y$.  Equivalently $Z$ is Kolchin dense in $X$ if $Z$ meets every Kolchin open subset of $X$. 

Let us note in passing that $DCF_{0}$ having quantifier elimination means that every definable subset of $X$ is a
Boolean combination of closed sets, equivalently  finite union of locally closed sets.

We will be concerned with the case when $X$ is a $DAG$ $G$ (in fact just a $LDAG$ suffices).  Assume $G$ defined over $K$. By $\omega$-stability, $G$ has a smallest
definale subgroup of finite index, which we denote by $G^{0}$. Then $G^{0}$ is the (Kolchin) irreducible component of $G$ which contains the identity, and the cosets of 
$G^{0}$ in $G$ are the other irreducible components.

Recall that a definable subset $Y$ of $G$ is said to be generic if finitely many translates of $Y$ cover $G$.  The same notion applies to (definable) torsors $X$ for $G$, namely a definable $Y\subseteq X$ is
generic if finitely many $G$-translates of $Y$ cover $X$. 
A complete type $p(x)\in S_{G}(K)$ is
generic if every formula (definable set) in $p$ is generic.  Likewise for complete types $p(x)\in S_{X}(K)$ where $X$ is a torsor for $G$ over $K$. 

A  basic fact from stable group theory is:
\begin{Fact} Let $\pi:G\to H$ be a definable surjective homomorphism, all defined over $K$ .  Let $Y\subseteq G$ be definable and generic in $G$. Then $\pi(Y)$ is generic
in $H$, and for some generic (in $H$) definable subset $Z$ of $\pi(Y)$, for any $b\in Z$, $\pi^{-1}(b)\cap X$ is generic in the coset $\pi^{-1}(b)$ of $ker(\pi)$. 
\end{Fact}

Finally the connection with the Kolchin topology is:
\begin{Fact} Let $Y$ be a definable subset of the $DAG$ $G$. Then $Y$ is generic if and only if $Y$ contains a nonempty (Kolchin) open subset of $G$.
Likewise for definable subsets of torsors. 
\end{Fact}

We conclude from Fact 1.23
\begin{Fact} A (abstract) subset $Z$ of $G$ is Kolchin dense in $G$ iff $Z$ meets every generic definable subset of $G$. Likewise for 
abstract subsets of torsors for $G$.
\end{Fact}

\section{Main results}
We now give proofs of results stated in the introduction.

\subsection{On arbitrary  $LDAG$'s and the proof of Proposition 1.1}

We first prove the finite-dimensional case, part (i) of Proposition 1.1.  By Lemma 1.15 it suffices to prove:


\begin{Proposition}
Suppose $K$ is a differential field which is algebraically closed and $PV$ closed. Let $G$ be a $fdLDAG$ over $K$. Then $G(K^{diff}) = G(K)$. And if $X$ is a 
differential algebraic torsor for $G$ over $K$ then also $X(K^{diff}) = X(K)$.
\end{Proposition}
\begin{proof}
By Theorem 1.1 of \cite{Pillay-PV} it suffices to prove the first part, that $G(K^{diff}) = G(K)$. 
This will be an elaboration on or adaptation of the proof of the finite-dimensional case of Theorem 1.1 in \cite{Pillay-PV} (see p. 813-814).  We will follow notation from there. 

\vspace{2mm}
\noindent
{\em Claim 1.}  Suppose $N$ is a $K$-definable normal subgroup of $G$, $S = G/N$,  $N(K^{diff}) = N(K)$ and $S(K^{diff}) = S(K)$. Then $G(K^{diff}) = G(K)$.
\begin{proof} By our assumptions, each coset of $N(K^{diff})$ in $G(K^{diff})$ is defined over $K$, and is also a torsor for $N$, 
so has a $K$-point (by Theorem 1.1 of \cite{Pillay-PV}). So as  $N(K^{diff}) = N(K)$, also $X(K^{diff}) =  X(K)$ for each coset 
$X(K^{diff})$ of $N(K^{diff})$ in $G(K^{diff})$ and  thus $G(K^{diff}) = G(K)$. 
\end{proof} 

\vspace{2mm}
\noindent
Note that if $G$ is finite then as $K$ is algebraically closed, $G(K^{diff})= G(K)$. So, by Claim 1 we may assume $G$ is connected.
Now $G$ has a connected, definable over $K$, solvable, normal subgroup $N$ and $G/N = S$ is ``semisimple", namely has no infinite definable normal abelian subgroup, and of course $S$ is connected. By Claim 1 it will be enough to prove that $N(K^{diff}) = N(K)$ and $S(K^{diff}) = S(K)$ which we do now.

\vspace{2mm}
\noindent
{\em Claim 2.}  $N(K^{diff}) = N(K)$. 
\begin{proof} As in \cite{Pillay-PV} $N$ is filtered by a chain of normal $K$-definable connected subgroups, such that each successive quotient is a (finite-dimensional) subgroup of $({\mathcal U},+)$ or of ${\mathcal U}^{\times}$.
By Claim 1, it suffices to prove Claim 2 when $N$ is one of these quotients. 
Suppose first $N \leq ({\mathcal U}, +)$. Then $N(K^{diff})$ is a finite-dimensional vector space over $C_{K}$. Then either directly or using Lemma 1.17, $N(K^{diff})$ generates over $K$ a $PV$ extension of $K$, so by our assumptions we get that $N(K^{diff}) = N(K)$.

Suppose now that $N\leq {\mathcal U}^{\times}$.  The logarithmic derivative takes $N$ onto a (finite-dimensional) $K$-definable subgroup of $({\mathcal U},+)$ for which the claim holds. And the kernel is just the multiplicative group of the constants for which the claim also holds. So by Claim 1, $N(K^{diff}) = N(K)$. 

This completes the proof of Claim 2.

\end{proof}

\vspace{2mm}
\noindent
{\em Claim 3.}  $S(K^{diff}) = S(K)$.
\begin{proof}  From the discussion in the proof of Case 1(b) of Theorem 1.1 in \cite{Pillay-PV} we see that $S$ is ``almost internal" to the constants. This means that that over additional parameters $S$ is in the algebraic closure of the constants. By Corollary 3.10 of \cite{Pillay-remarks}, $S$ is internal to the constants. By Lemma 1.17 again it follows that $S(K^{diff})$ generates over $K$ a $PV$ extension of $K$, so by our assumptions, $S(K^{diff}) = S(K)$.
\end{proof}

This completes the proof of Proposition 2.1.

\end{proof}

We now move on to part (ii) of Proposition 1.1. As above it suffices to prove:
\begin{Proposition} Suppose that the differential field $K$ is algebraically closed and $PV$-closed. Let $G$ be a $LDAG$ over $K$. Then
$G(K)$ is Kolchin dense in $G$. 
\end{Proposition}
\begin{proof}  We start with a Kolchin-dense version of the first claim from the previous proposition.

\vspace{2mm}
\noindent
{\em Claim 4.} Suppose $N$ is a normal subgroup of $G$ definable over $K$ and let $S = G/N$. Suppose that $N(K)$ is Kolchin-dense in $N$, and 
$S(K)$ is Kolchin dense in $S$. Then $G(K)$ is Kolchin dense in $G$.
\newline
\begin{proof} Let $\pi:G\to S$. By Fact 1.24 we have to show that if  $X$ be a generic definable subset of $G$ then  $X(K)\neq \emptyset$. 
By Fact 1.22 there is a generic (in $S$) definable subset $Z$ of $\pi(X)$ such that for any $b\in Z$, $\pi^{-1}(b)\cap X$ is generic in the coset $\pi^{-1}(b)$
of $N$.  As $S(K)$ is assumed to be Kolchin dense in $S$, by Fact 1.24 there is $b\in Z(K)$, whereby $\pi^{-1}(b)$ is a torsor over $K$ for $N$. By assumption,
$N(K)$ is Kolchin dense in $N$. Also by Theorem 1.1 of \cite{Pillay-PV}, $\pi^{-1}(b)$ has a $K$-point, so $\pi^{-1}(b)(K)$ is also Kolchin dense
in $\pi^{-1}(b)$.  By Fact 1.24 again $\pi^{-1}(b)\cap X$ has a $K$-point. In particular $X$ has a $K$-point, which is what we had to prove. 

\end{proof}

We will be using the material around the proof of Case (2) of Theorem 1.1 from \cite{Pillay-PV}. See that paper for appropriate references. 
Note that by Proposition 2.1, Proposition 2.2 holds when $G$ is finite-dimensional.  Hence by Claim 4, it suffices to prove Proposition 2.1 when $G$ is what
we call in \cite{Pillay-PV} $1$-connected, meaning having no proper definable (normal) subgroup $H$ such that $G/H$ is finite-dimensional.
Claim 4, together with the analysis in the proof of Case (2) of Theorem 1.1 in \cite{Pillay-PV}, reduces to the case where $G$ is the additive group
of ${\mathcal U}$, the multiplicative group of ${\mathcal U}$, or the ${\mathcal U}$-points of a simple algebraic group  defined over $K$.

\vspace{2mm}
\noindent
{\em Claim 5.}  If $G$ is either the additive or multiplicative group of ${\mathcal U}$, then $G(K)$ is Kolchin dense in $G$. 
\begin{proof} By Fact 1.24, we have to show that every generic definable subset $X$ of $G$ meets $G(K)$. On general grounds we may assume $X$ 
to be defined over $K^{diff}$. Then $X$ is of the form $G\setminus Y$ where $Y$ is of finite order (say $n$) and defined over a finite tuple ${\bar a}$
from $K^{diff}$.   Let $t\in K$ be such that $t'=1$. Let $K_{0} = {\mathbb Q}^{alg}(t)$ a differential subfield of $K$. Let $C$ be a differential transcendence
basis of $K$ over $K_{0}$, and let $K_{1} = K_{0}\langle C\rangle$.  Notice that the order of ${\bar a}$ over $K_{1}$ is finite, say $m$. Hence for every $d\in Y$ the order
of $d$ over $K_{1}$ is at most $m+n$.  On the other hand, by \cite{Mitschi-Singer}, Kolchin's primitive element theorem \cite{Kolchin}, and our assumptions that $K$ is $PV$-closed,
for each $r$ there is $e\in K$ with the order of $e$ over $K_{0}$ finite and $\geq r$. Then also the order of $e$ over $K_{1}$ is $\geq r$. But then $e\notin Y$,
so $e\in X$, whereby $X(K)$ is nonempty, as required. 
\end{proof}

Now suppose that $G$ is a simple algebraic group over $K$, in particular linear and connected. In that case, $G$ is, as an algebraic variety over $K$, a rational variety, namely
birational over (the algebraically closed field) $K$ to some affine space ${\mathcal U}^{n}$.  (See \cite{Borel}, Remark 14.14.) Hence we are reduced to proving that $K^{n}$ is Kolchin
dense in ${\mathcal U}^{n}$.  This is an adaptation of the proof of Claim 5.  We will consider a special case where the Kolchin open subset $X$ of ${\mathcal U}^{n}$ is given by
a differential polynomial inequation $P(x_{1}, x_{2},..., x_{n-1}, y) \neq 0$ where $P$ is over $K^{diff}$.  (And leave the general case to the reader.) Let $K_{0}$ and $K_{1}$ be as in the proof of Claim 5.
Let $a_{1},..., a_{n-1}\in K_{0}$.  The proof of Claim 5 gives elements $e$ in $K$ of arbitrarily large order over $K_{1}$, and therefore $e$ such that $P(a_{1},..,a_{n-1},e)\neq 0$. 

\vspace{2mm}
\noindent
This completes the proof of Proposition 2.2, and thus of Proposition 1.1.
\end{proof}

We mentioned in the introduction that  Proposition 1.1 (ii) yields an equality of differential Galois cohomology groups at the level of definable cocycles.
Let us be more precise on what is meant.  The proof is obvious given material in \cite{Pillay-Galois}. 
In Subsection 1.1 we gave a precise definition of $H^{1}_{\partial}(Aut_{\partial}(K^{diff}/K), G(K^{diff}))$ for $G$ a $DAG$ over the differential field $K$ (which is also called just 
$H^{1}_{\partial}(K, G)$ following Kolchin). But we can adapt to this to any nice differential field extension of $K$ in place of $K^{diff}$.
In particular we can consider $K^{PV_{\infty}}$.  By a $K$-definable cocycle from ${\mathcal G}_{1} = Aut_{\partial}(K^{PV_{\infty}}/K)$ to $G(K^{PV_{\infty}})$ we mean a crossed
homomorphism (as before) $f$ from ${\mathcal G}_{1}$ to $G(K^{PV_{\infty}})$ such that for some $K$-definable function $h(-,-)$ and ${\bar a}$ from 
$K^{PV_{\infty}}$ we have that $f(\sigma) = h({\bar a}, \sigma({\bar a}))$ for all $\sigma \in {\mathcal G}_{1}$. Two such cocycles $f, g$ are cohomologous if for
some $b\in G(K^{PV_{\infty}})$, $f(\sigma) = b^{-1}g(\sigma)\sigma(b)$ for all $\sigma\in {\mathcal G}_{1}$. $H^{1}_{\partial}(Aut_{\partial}(K^{diff}/K), G(K^{diff}))$
is the collection of equivalence classes.  Let us first note that on general model-theoretic grounds, every automorphism of $K^{PV_{\infty}}$ over $K$ extends to an automorphism of 
$K^{diff}$ over $K$. 
\begin{Proposition} 
Let $G$ be a $fdLDAG$ over $K$, and let $f$ be a $K$-definable cocycle from $Aut_{\partial}(K^{diff}/K)$ to $G(K^{diff})$. Then 
\newline
(i) for any $\sigma\in Aut_{\partial}(K^{diff}/K)$, $f(\sigma)$ depends only on $\sigma|K^{PV_{\infty}}\in Aut_{\partial}(K^{PV_{\infty}}/K)$.
\newline
(ii) There is a $K$-definable function $h(-,-)$ and tuple ${\bar a}$ from $K^{PV_{\infty}}$ such that $f(\sigma) = h({\bar a}, \sigma({\bar a}))$ for all 
$\sigma\in {\mathcal G}$ (or equivalently ${\mathcal G}_{1}$.

\end{Proposition}

\subsection{On fdLDAG's internal to the constants}
We want to find some tighter connections between torsors for $fdLDAG$'s $G$ over $K$ and $PV$ extensions, in the case that $G$ is internal to the constants.  This section was partly motivated by exploring  differential field versions of Serre's theorem referred to in the introduction.  We prove Propositions 1.2 and 1.3 from the introduction, which we will restate.

\begin{Proposition}  Let  $\partial Y_{1} =A_{1}Y_{1}$, and $\partial Y_{2} = A_{2}Y_{2}$ be linear differential equations over $K$, where $Y_{1}$, $Y_{2}$ are  column vectors of indeterminates of length $n_{1}$, $n_{2}$.
Let  $G_{i} < GL_{n_{i}}(K^{diff})$ be the intrinsic differential Galois group of the equation $\partial Y_{i} = A_{i}Y_{i}$
and let $(G_{i}, X_{i})$ be the corresponding torsor defined over $K$, for i=1,2., as in 1.9 (i).  Let $L_{1}, L_{2}$ be the corresponding  Picard-Vessiot extensions of $K$.
\newline
Then $L_{1} = L_{2}$ if and only if  $(G_{1}, X_{1})$, and $(G_{2}, X_{2})$ are isomorphic over $K$ as (differential algebraic) torsors in the sense of Definition 1.4 (i). 
\end{Proposition} 
\begin{proof}  Suppose that the right hand side holds. So there are ${\bar b}_{1}\in X_{1}$ and ${\bar b}_{2}\in X_{2}$ which are interdefinable over $K$, so the differential field generated by $K$ and ${\bar b}_{1}$ is equal to the differential field generated by ${\bar b}_{2}$. But these are precisely  $L_{1}$ and $L_{2}$, so we get the left hand side.

Now suppose that $L_{1} = L_{2} = L$ say. Let again ${\bar b}_{1}\in X_{1}$ and ${\bar b}_{2}\in X_{2}$.  So ${\bar 
b}_{1}$ and ${\bar b}_{2}$ are interdefinable over $K$ (even in the field language). As $X_{1}$ is the set of realizations 
of $tp({\bar b}_{1}/K)$  (in $K^{diff}$) and $X_{2}$ the set of realizations of $tp({\bar b}_{2}/K)$ (in $K^{diff}$), we 
obtain a definable over $K$  bijection $j:X_{1} \to X_{2}$ the graph of which is precisely the set of realizations of $tp({\bar 
b}_{1}, {\bar b}_{2}/K)$ (in $K^{diff}$).  In particular $j$ is invariant under $Aut(L/K)$. 

Now from 1.9 (i),  we have an isomorphism $\rho_{1}: Aut(L/K) \to  G_{1}$, where $\rho_{1}(\sigma) = \sigma({\bar b}_{1}){\bar b}_{1}^{-1}$   (which does not depend on the choice of ${\bar b}_{1}\in X_{1}$). And similarly $\rho_{2}:Aut(L/K) \to G_{2}$ with analogous properties. 

So $i = \rho_{2}\circ \rho_{1}^{-1}$ is an isomorphism (of abstract groups) between $G_{1}$ and $G_{2}$.

\vspace{2mm}
\noindent
{\bf Claim I.}  $i$ is definable over $K$.
\newline
{\em Proof of Claim I.} 
We first show that $i$ is definable over ${\bar b}_{1}$.  Given $g\in G_{1}$, let $\rho_{1}(\sigma) = g$, so $\sigma({\bar b}_{1}) = g{\bar b}_{1}$.  But as $j$ is invariant under $\sigma$, $\sigma({\bar b}_{2}) = j(g{\bar b}_{1})$ which also equals $\rho_{2}(\sigma){\bar b}_{2}$.   The end result is that $i(g) = j(g({\bar b}_{1})(j({\bar b}_{1}))^{-1}$, so $i$ is definable over ${\bar b}_{1}$.   
But the same computation and formula work for any element of $X_{1}$ in place of ${\bar b}_{1}$.  This proves Claim I.
\qed

\vspace{2mm}
\noindent
{\bf Claim II.}  The pair $(i,j)$ gives an isomorphism over $K$ of the differential algebraic torsors $(G_{1}, X_{1})$ and $(G_{2}, X_{2})$.
\newline
{\em Proof of Claim II.}   This is immediate from the definition of $i$ and $j$, as we know that $j$ is an isomorphism between $X_{1}$ and $X_{2}$ as abstract torsors for $Aut(L/K)$. 
\qed

\end{proof}

A natural equivalence relation on the collection of linear DE's over $K$, $\partial Y = AY$ where $Y$ is a column vector of unknowns (of varying size), 
is that two such linear DE's define the same $PV$ extension of $K$. As we saw, each such an equation gives rise to a left torsor $(G,X)$.
 So the proposition says that fixing a $fdLDAG$, $G$,  over $K$, the collection of equivalence classes of linear DE's over $K$ 
with intrinsic Galois group isomorphic over $K$ to $G$ embeds in the collection of isomorphism classes of left torsors $(G,X)$ over $K$ 
 (which recall is not the same as $H^{1}_{\partial}(K,G)$).
What about in the other direction?

The following is basically an adaptation of the proof of the ``less trivial" direction of Serre's theorem Proposition 8 in \cite{Serre}, to the differential context. 

\begin{Proposition} Let $X$ be a right torsor for a $fdLDAG$ $G$, all over $K$, where $G$ is internal to the constants.  
Let $K_{G}$ be the differential field generated over $K$ by $G(K^{diff})$ and $K_{X}$ the differential field generated over $K$ by $X(K^{diff})$. 
Then $K\leq K_{G} \leq K_{X}$, $K_{X}$ is a PV extension of $K$, so also a PV extension of $K_{G}$. Moreover, the differential Galois group of $K_{X}$ over $K_{G}$ (intrinsic or extrinsic) definably, 
over $K_{X}$, embeds in $G$. 
\end{Proposition}
\begin{proof} By Lemma 1.17, $K_{G}$ is a $PV$ extension of $K$. Note that $K_{G}\leq K_{X}$ (as any element of $G(K^{diff})$ is obtained from
some two elements of $X(K^{diff})$).  As $G$ is internal to the constants, so is $X$, so by Lemma 1.21 $K_{X}$ is a strongly normal 
extension of $K$, and so also a strongly normal extension of $K_{G}$.  Fix $b\in X$ and note that $K_{X}$ is generated over $K_{G}$ by $b$ (as a differential field).  We consider now 
$Gal(K_{X}/K_{G}) = Aut_{\partial}(K_{X}/K_{G})$.  For $\sigma\in Gal(K_{X}/K_{G})$ there is a unique $g_{\sigma}\in G(K^{diff}) = G(K_{G})$ such that $\sigma(b) = bg_{\sigma}$, 
moreover $\sigma$ is determined by $g_{\sigma}$. This determines a definable over $K_{X}$ embedding of the extrinsic differential Galois group of $K_{X}$ over $K_{G}$ into $G$. 
By Lemma 1.17, there is a $K_{G}$-definable isomorphism of $G(K^{diff})$ with an algebraic subgroup of $GL_{n}(C_{K})$ (some $n$), hence the extrinsic differential  Galois
group of $K_{X}$ over $K_{G}$ is an algebraic subgroup of $GL_{n}(C_{K})$. 

But then the extrinsic differential Galois group of $K_{X}$ over $K$ is an extension of a linear algebraic group by a linear algebraic group (inside $C_{K})$), hence must also be linear. 
By Fact 1.20, $K_{X}$ is a $PV$ extension of $K$.
\end{proof}

\subsection{Notions of smallness}
 In this section $K$ is still a differential field with $C_{K}$ algebraically closed.

 We let $G$ be  a $fdLDAG$ over $K$ which is (almost) internal to the constants. Let $K_{G}$ be the $PV$ extension of $K$ generated by $G(K^{diff})$. 
As mentioned earlier, finiteness conditions on either $PV$ extensions of $K$ with Galois group $G$, or on $H^{1}_{\partial}(K,G)$ are too much to expect outside very special situations. 
We tentatively suggest some weaker conditions which will be explored in the future.  The basic smallness condition is that a collection of Picard-Vessiot extensions of $K$ corresponding to
a given condition is contained in a Picard-Vessiot extension. 

Here are four conditions on $G$.

\noindent
1) The collection of $PV$ extensions $L$ of $K$ with intrinsic Galois group $K$-definably isomorphic to $G$  lives inside (generates) a PV extension of $K$.
\newline
\noindent
2) The collection of $PV$ extensions $L$ of $K$ with intrinsic or extrinsic Galois group definably isomorphic to $G$ over $L$  lives inside (or generates) a PV extension of $K$.
\newline
\noindent
3) The collection of $K^{diff}$ points of left (right) torsors for $G$ over $K$ live inside (generate) a PV extension of $K$.
\newline
\noindent
4)  The collection of $PV$ extensions $L$ of $K$ containing $K_{G}$ such that the (intrinsic or extrinsic) Galois group of $L$ over $K_{G}$ definably over $L$ embeds in $G$, is contained in (generates) a $PV$ extension of K. 

We have various implications:

\vspace{2mm}
\noindent
(2) implies (1) is immediate.  
\newline
(3) implies (1) Lemma 1.9.
\newline
(4) implies (3) is Proposition 2.5. 

\vspace{5mm}
\noindent
A final remark concerns the interpretation of conditions such as (3).  It will mean that the constrained cohomology set $H^{1}_{\partial}(K,G)$ is really an object 
belonging to a first cohomology set related to algebraic actions of linear algebraic groups on each other.
We give a brief explanation.

Let us assume (3). Let $L$ be a $PV$ extension of $K$ such that for every torsor $X$ for $G$ over $K$, $X(K^{diff}) = X(L)$.
In particular $G(K^{diff}) = G(L)$.   So $H^{1}_{\partial}(K,G)$ =  $H^{1}_{\partial}(Gal(L/K), G(L))$.
Let $G_{1}$ be the extrinsic Galois group of $L$ over $K$. So $G_{1}$ is a linear algebraic group in the sense of the algebraically closed field
$C_{K}$. Let us write $K_{0}$ for this algebraically closed field $C_{K}$.
We also know (Lemma 1.17) that $G$ is definably isomorphic to a linear algebraic group $G_{2}$ in the sense of $K_{0}$.
Now the action of $Gal(L/K)$ on $G(L)$ via differential field automorphisms, yields an algebraic action of $G_{1}$ on $G_{2}$, or if you want a rational 
homomorphism from $G_{1}$ to $Aut(G_{2})$ where $Aut(G_{2})$ means automorphisms in the sense of algebraic groups.
A definable cocycle $f:Aut(L/K) \to G(L)$ will yield a morphism $f':G_{1} \to G_{2}$ (of algebraic varieties over $K_{0}$)  such that $f'(\sigma\tau) = f'(\sigma)\sigma(f'(\tau))$, and
define two such morphisms $f'$, $h'$ to be cohomologous if for some $b\in G_{2}$, $f'(\sigma) = b^{-1}h'(\sigma)\sigma(b)$ for all $\sigma\in G_{1}$.

Hence $H^{1}_{\partial}(K,G)$ can be described as the collection of algebraic cocycles from $G_{1}$ to $G_{2}$ (with respect to a given algebraic action of $G_{1}$ 
on $G_{2}$) up to being cohomologous, so an object living entirely in algebraic geometry.



\end{document}